\documentclass{amsart}

\usepackage{amsfonts,amsmath,amsthm}
\usepackage{enumitem} 
\usepackage{booktabs}

\usepackage{bm}

\usepackage{mathtools}    
\DeclarePairedDelimiterX\setc[2]{\{}{\}}{\,#1 \;\delimsize\vert\; #2\,}

\usepackage[dvipsnames]{xcolor}

\newcommand{\R}{\mathbb{R}}
\newcommand{\N}{\mathbb{N}}
\newcommand{\B}{\mathbb{B}}

\renewcommand{\S}{\mathbb{S}}
\newcommand{\Sm}{\mathbb{S}^m}
\newcommand{\Sn}{\mathbb{S}^n}

\newcommand{\blk}{\overline{\lambda}_k}
\newcommand{\blo}{\overline{\lambda}_1}
\newcommand{\bla}{\overline{\lambda}_2}
\renewcommand{\l}{\lambda}

\newtheorem{defi}{Definition}[section] 
\newtheorem{thm}[defi]{Theorem}

\newtheorem{rem}[defi]{Remark}
\newtheorem{prop}[defi]{Proposition}
\newtheorem{lemma}[defi]{Lemma}
\newtheorem{cor}[defi]{Corollary}

\title[Conformal bounds for the second eigenvalue of the Laplacian]{Upper bounds for the second nonzero eigenvalue of the Laplacian via folding and conformal volume}

\author{{Mehdi} {Eddaoudi}}\email{mehdi.eddaoudi.1@ulaval.ca}\address{{D\'epartement de math\'ematiques et de statistique}, {Pavillon Alexandre-Vachon}, {Universit\'e Laval}, {Qu\'ebec QC}, {G1V 0A6}, {Canada}}
\author{{Alexandre} {Girouard}}\email{alexandre.girouard@mat.ulaval.ca}\address{{D\'epartement de math\'ematiques et de statistique}, {Pavillon Alexandre-Vachon}, {Universit\'e Laval}, {Qu\'ebec QC}, {G1V 0A6}, {Canada}}

\begin{document}

\begin{abstract}
We prove an upper bound for the volume-normalized second nonzero eigenvalue of the Laplace operator on closed Riemannian manifold, in terms of the conformal volume. This bound provides effective upper bound for a large class of manifolds, thereby generalizing several known results.
\end{abstract}

\maketitle

\section{\bf Introduction}
Let $(M,g)$ be a closed Riemannian manifold of dimension $m\geq 2$. The spectrum of the Laplace operator $\Delta_g$ is discrete, non-negative and unbounded:
$$0=\lambda_0(M,g) < \lambda_1(M,g) \leq \lambda_2(M,g) \leq \cdots\nearrow\infty,$$
where each eigenvalue is repeated according to its multiplicity. Let 
$\mathcal{M}(M)$ be the set of all Riemannian metric $g$ on $M$ such that $\text{vol}(M,g)=1$. One of the central theme of spectral geometry is to bound the individual eigenvalue functionals $\lambda_k:\mathcal{M}(M)\to\R$. By homogeneity this problem is equivalent to that of bounding the scale-invariant functionals
$$\blk(M,g):=\lambda_k(M,g)\text{vol}(M,g)^{2/m}$$
among all Riemannian metrics $g$ on $M$.

\subsubsection*{\bf Surfaces}
For surfaces, this problem goes back to Hersch~\cite{Hersch}, who proved that the canonical round metric $g_{\S^2}$ uniquely maximizes $\lambda_1(\S^2,g)$ among all metrics $g$ that have area equal to $4\pi$. Equivalently, $\blo(\S^2,g)\leq 8\pi$, with equality if and only if $g$ is isometric to $kg_{\S^2}$, for some constant $k>0$. For closed orientable surfaces $M$ of genus $\gamma$, upper-bounds were obtained by Yang and Yau~\cite{YangYau} and then slightly improved by El Soufi and Ilias~\cite{ElSoufiIlias1984}:
\begin{gather}\label{ineq:YangYau}
    \blo(M,g)\leq8\pi\left\lfloor\frac{\gamma+3}{2}\right\rfloor.
\end{gather}
See~\cite{Ros2022, Ros2023, KarpukhinVinokurov2022} for recent improvements.
This bound is attained on the sphere: this is Hersch result. In their recent work~\cite{NayataniShoda2019}, Nayatani and Shoda showed that it is also attained on orientable surfaces of genus $\gamma=2$. However Karpukhin~\cite{Karpukhin2019} proved that inequality~\eqref{ineq:YangYau} is strict for all other genera $\gamma$. Sharp bounds are also known for the torus and the Klein bottle~\cite{Nadirashvili1996, CiKaMe2019} as well as the projective plane~\cite{Li-Yau}. 

For arbitrary $k\geq 1$, Karpukhin, Nadirashvili, Penskoi and Polterovich~\cite{KNPP2} proved that 
\begin{gather}\label{ineq:KNPPSurfaces}
    \blk(M,g)\leq8\pi k\left\lfloor\frac{\gamma+3}{2}\right\rfloor.
\end{gather}
On the sphere $\S^2$, the same authors~\cite{KNPP2021} proved that inequality~\eqref{ineq:KNPPSurfaces} is sharp for all $k$, and strict for $k\geq 2$.
Sharp upper bounds are also known for the projective plane~\cite{NadirashviliPenskoi2018, Karpukhin2021}:
$$\blk(\mathbb{RP}^2,g)\leq 4\pi(2k+1),$$
with strict inequality if $k\geq 2$.


\subsection{\bf The conformal volume of Li and Yau}
The situation is completely different for manifolds $M$ of dimension $m\geq 3$. Indeed in this case the functional $\lambda_1:\mathcal{M}(M)\to\R$ is not bounded above. This was proved in full generality by Colbois and Dodziuk~\cite{colbois-Dodziuk}, following work of Bleecker~\cite{Bleecker1983} and others. However, by further restricting $\lambda_1$ to a conformal class $C=[g_0]$, El Soufi and Ilias~\cite{ElSoufiIlias1986} were able to prove that
$$\lambda_1:\mathcal{M}(M)\cap C\to\R$$
is bounded above, in terms of the \emph{conformal volume} $V_c(M,C)$. This conformal invariant was introduced by Li and Yau in their influential paper~\cite{Li-Yau}. Let us briefly recall its definition. Given a conformal immersion $\phi:M\to\S^n\subset\R^{n+1}$, let
$$V_c(n,\phi)=\sup_{\tau\in\text{Aut}(\S^n)}\text{vol}(M,(\tau\circ\phi)^{\star}g_{\S^n}),$$
and define the \emph{$n$-conformal volume} of $(M,C)$ to be
$$V_c(n,M,C)=\inf_{\phi:M\to\S^n}V_c(m,\phi),$$
where the infimum is over all conformal immersions. Observe that $V_c(n,M,C)$ is well-defined as soon as $n$ is large enough, since a combination of the Nash embedding theorem with a stereographic projection provides a conformal embedding of $(M,C)$ in a sphere. The \emph{conformal volume of $(M,C)$} is 
$$V_c(M,C)=\lim_{n\to\infty}V_c(n,M,C).$$
See~\cite{Li-Yau,ElSoufiIlias1984} for properties of the conformal volume.

Li and Yau~\cite{Li-Yau} proved upper bounds for $\lambda_1$ on surfaces in term of the conformal volume, thereby improving inequality~\eqref{ineq:YangYau}. Soon after, this was generalized to manifolds of arbitrary dimensions by El Soufi and Ilias~\cite{ElSoufiIlias1986}, who proved the following fundamental theorem.
\begin{thm}\label{thm:ElSoufiIlias}
Let $M$ be a closed $m$-dimensional manifold. For each conformal class $C$ on $M$ and each Riemannian metric $g\in C$ that admits a conformal immersion $\phi:(M,C)\to\S^n\subset\R^{n+1}$,
\begin{gather}\label{ineq:ElSoufiIlias}
    \blo(M,g)\leq mV_c(n,M,C)^{2/m},
\end{gather}
with equality if and only if there exists a minimal immersion $\phi:M\to\S^n$ such that
\begin{gather}\label{eq:eigenphi}
 -\Delta_g\phi=\lambda_1(M,g)\phi,
\end{gather}
and $\phi^{\star}g_{\S^n}=kg$ for some constant $k>0$. 
\end{thm}
Not only does Theorem~\ref{thm:ElSoufiIlias} provide an interesting upper bound, but the equality case is one of the main tool available for the explicit computation of the conformal volume of various manifolds. In particular, it is known that irreducible homogeneous manifolds admits minimal immersions by first eigenfunctions. So do all strongly harmonic manifolds. For more on this, see~\cite{ElSoufiIlias1984,Besse1978}.

For arbitrary index $k$, Korevaar~\cite{korevaar1993upper} was then able to prove that each functional
$\lambda_k:\mathcal{M}\cap C\to\R$
is bounded above, however without providing effective bounds. One should also see the work of Grigor'yan, Netrusov and Yau~\cite{GNY2004} where the method of Korevaar is revisited and developed in the general setting of Dirichlet energy forms, providing a flexible tool that is applicable in various contexts. To be slightly more precise, the results of~\cite{korevaar1993upper,GNY2004} imply the existence of a constant $K(C)$ depending on the conformal class $C$ on $M$ such that each $g\in C$ satisfies
\begin{gather}\label{ineq:Korevaar}
    \blk(M,g)\leq K(C) k^{2/m}
\end{gather}
for each index $k\in\N$. When we say that the result of Korevaar is not effective, we mean that the value of the constant $K(C)$ is not known, or that it is not expressed in terms of well understood quantities. 
This situation was recently improved by Kokarev~\cite{Kokarev}. He proved that one can take $K(C)=C(m,n)\mathcal{V}_c(M,n)$ in inequality~\eqref{ineq:Korevaar}, as long as $n$ is large enough for $(M,C)$ to admit a conformal immersion $\phi:(M,C)\to\S^n\subset\R^{n+1}$:
\begin{gather}\label{ineq:Kokarev}
\blk(M,g) \leq C(m,n)V_c(n,M,C)k^{2/m}.    
\end{gather}
It follows from Theorem~\ref{thm:ElSoufiIlias} that for $k=1$, one can take $C(m,n)=m$.
We will see shortly that the main result of the present paper shows that for $k=2$, one can use 
$$C(m,n)=2^\frac{2}{m}m.$$
One should also see~\cite{Hassannezhad2011} for related results.

\subsection{\bf The main result and its consequences}
Our main result provides an effective bound for $\lambda_2:\mathcal{M}\cap C\to\R$ that is similar the above inequality~\eqref{ineq:ElSoufiIlias}, in terms of the dimension and of the conformal volume of $(M,C)$. 
\begin{thm}\label{thm:thmprincipal}
    Let $M$ be closed manifold of dimension $m$. For each conformal class $C$ on $M$ that admits a conformal immersion $\phi:(M,C)\to\S^n\subset\R^{n+1}$, the following inequality holds for each $g\in C$,
    \begin{gather}\label{ineq:thmprincipal}
        \overline{\lambda}_2(M,g)<2^\frac{2}{m}m V_c(n,M,C)^{2/m}.
    \end{gather}
\end{thm}
The main interest of Theorem~\ref{thm:thmprincipal} is that it provides new effective upper bounds for a large class of manifolds of arbitrary dimension. For instance, when considering surfaces, Theorem~\ref{thm:thmprincipal} is already a known result in Karpukhin and Stern~\cite{karpukhinStern2023}[proposition 1.20].

Furthermore, whenever there exists a branched cover  of degree $d$ from a surface $M$ into the sphere $\S^2$, we have for all $n$,

\[
    4\pi \leq V_c(n,M,C) \leq d V_c(n,\S^2)= 4\pi d.
\]
More properties of the conformal volume can be found in \cite{ElSoufiIlias1984, Kokarev, Li-Yau}.
As this degree $d$ can be bounded from above by $\left\lfloor\frac{\gamma+3}{2}\right\rfloor$, we obtain the inequality \eqref{ineq:KNPPSurfaces} for $k=2$.

\begin{cor}
Let $(M,g)$ be a closed orientable Riemannian surface of genus $\gamma$. Then
    $$\overline{\lambda}_2(M,g)< 16\pi\left\lfloor\frac{\gamma+3}{2}\right\rfloor.$$
\end{cor}

For closed surfaces of $\R^n$, thanks to an inequality obtained by Li and Yau \cite{Li-Yau} between the conformal volume $V_c(n,M,C)$ and the norm of the mean curvature vector $H$ of the immersion,
$$V_c(n,M,C) \leq \int_M \|H\|^2dv_g,$$
we obtain the following corollary which resembles a sort of Reilly inequality for $\bla$.

\begin{cor}
    Let $M$ be a closed surface in $\mathbb{R}^n$, and let $g$ be the induced metric from $\R^n$, then

      $$\bla < 4\int_M \|H\|^2dv_g.$$
      
\end{cor}

Any manifold $(M,g)$ for which the equality case in Theorem~\ref{thm:ElSoufiIlias} is realized also lend itself to an interesting eigenvalue comparison.
\begin{cor}\label{Coro:lambda2homo}
    Let $(M,g_0)$ be an irreducible homogeneous space of dimension $m$. For each $g\in[g_0]$,
    \begin{gather}\label{ineq:lambda2homo}
        \overline{\lambda}_2(M,g)<2^\frac{2}{m}\blo(M,g_0).
    \end{gather}
\end{cor}
Because it is instructive and very short, we give the proof here.
\begin{proof}[Proof of Corollary~\ref{Coro:lambda2homo}]
A well-known result by Takahashi~\cite{takahashi1966} states that any irreducible compact homogeneous space $(M,g_0)$, admits a minimal immersion $\phi:(M,g_0)\to\S^n$ by its first eigenfunctions. This means that the equality case in Theorem~\ref{thm:ElSoufiIlias} is realized:
$$\blo(M,g_0)=mV_c(n,M,C)^{2/m},$$
where $n+1$ is the multiplicity of $\lambda_1(M,g_0)$. 
The result is then obtained by substitution into~\eqref{ineq:thmprincipal}.
\end{proof}

\subsection{\bf Explicit upper bounds for \texorpdfstring{$\bla$}{TEXT} of various spaces}
Through Corollary~\ref{Coro:lambda2homo}, Theorem~\ref{thm:thmprincipal} leads to explicit upper bounds for the second volume-normalized eigenvalue $\bla$ of any compact manifold $(M,g)$ that admits a minimal immersion $\phi:M\to\Sn$ by first eigenfunctions. In order to present some of them, we now recall some basic facts which will set the notation. 
For the sake of simplicity, whenever a manifold $M$ has a known canonimal metric $g_0$, we will remove it from the notation and write $\lambda_k(M)=\lambda_k(M,g_0)$. For details on homogeneous representation and volumes, see the book~\cite{Besse1978} or the shorter paper~\cite{BoSuTi2003}.
The results of this section are summarized in Table~\ref{tab:upperboundhomo}, in which the last two columns are the upper bounds provided by Theorem~\ref{thm:ElSoufiIlias} and Corollary~\ref{Coro:lambda2homo}.

\begin{table}
\small
\begin{center}

\begin{tabular}{llll}

\toprule

Manifold&Dimension&Theorem~\ref{thm:ElSoufiIlias}&Theorem~\ref{thm:thmprincipal} \\

\midrule

$\bm{(M,g_0)}$& $\bm{m=}$\bf dim$\bm{(M)}$& $\bm{\blo\leq}$ &$\bm{\bla<}$ \\[.15cm]

$\S^m$  &$m$&$m\omega_m^{2/m}$&$m(2\omega_m)^{2/m}$  \\[.1cm]

$\mathbb{P}^m(\mathbb{R})$&$m$&$2(m+1)(\frac{\omega_m}{2})^{2/m}$&$2(m+1)\omega_m^{2/m}$ \\[.1cm]

$\mathbb{P}^d(\mathbb{C})$&$m=2d$&$\frac{4\pi(d+1)}{d!^{1/d}}$&$\frac{2^{1/d}4\pi(d+1)}{d!^{1/d}}$ \\[.1cm]

$\mathbb{P}^d(\mathbb{H})$ &$m=4d$&  $\frac{8\pi(d+1)}{(2d+1)!^{1/2d}}$
& $\frac{2^{1/2d}8\pi (d+1))}{(2d+1)!^{1/2d}}$ \\[.1cm]

$\mathbb{P}^2(\mathbb{O})$&$m=16$&$48\pi(\frac{6}{11!})^{1/8}$& $48 \pi(\frac{12}{11!})^{1/8}$\\[.1cm]

$\mathbb{S}^p\left(\sqrt{\frac{p}{m}}\right)\times \mathbb{S}^q\left(\sqrt{\frac{q}{m}}\right)$&$m=p+q$ & $(p^p q^q)^{1/p+q}(\omega_p\omega_q)^{2/p+q}$  &  $(p^p q^q)^{1/p+q}(2\omega_p\omega_q)^{2/p+q}$\\[.1cm]

 $\mathbb{T}^2_{\text{eq}}$ & $m=2$ & $8 \pi^2 \sqrt{3}/3$ & $16 \pi^2 \sqrt{3}/3$\\[.1cm]

 $\mathbb{K}$  & $m=2$ & $12 \pi E( \frac{2\sqrt{2}}{3})$&  $24 \pi E( \frac{2\sqrt{2}}{3})$\\
\bottomrule

\end{tabular}

\end{center}

\vspace*{6pt}
 
\caption{\label{tab:upperboundhomo} Upper bounds for $\blo$ and $\bla$ on some homogeneous spaces.}

\end{table}

%
   
%

\subsubsection*{Spheres and Real Projective Spaces}

The standard round sphere $\Sm \subset \R^{m+1}$ has volume $\omega_m = \frac{2 \pi ^\frac{m+1}{2}}{\Gamma(\frac{m+1}{2})}$, where $\Gamma$ is the gamma function. The first non-zero eigenvalue is $\lambda_1(\Sm) = m$. The next distinct eigenvalue is $2(m+1)$.

The real projective space of dimension $m$, $\mathbb{RP}^m$, is the quotient of $\Sm$ by the antipodal map. Its volume is $\omega_m/2$, and it follows by symmetry that $\lambda_1(\mathbb{RP}^m) = 2(m+1)$.

By applying Theorem \ref{thm:thmprincipal} to the sphere $\Sm$ and the projective space $\mathbb{RP}^m$, we recover the recent results of Hanna Kim~\cite{Kim,Kimprojectivespace}.
\begin{cor}\label{Kimresultsphere}
Let $\omega_m$ be the volume of the standard round sphere $\S^n \subset \R^{m+1}$. Let $C_{\Sm}$ be the corresponding conformal class. For each $g \in C_{\S^n}$,
\begin{gather}\label{ineq:KimSphere}
    \bla(\Sm, g) < 2^{\frac{2}{m}} \blo(\Sm) = m(2\omega_m)^{2/m}.
\end{gather}
\end{cor}

\begin{cor}
Let $C_{\mathbb{P}^m(\mathbb{R})}$ be the canonical conformal structure on the projective space $\mathbb{P}^m(\mathbb{R})$. For each $g \in C_{\mathbb{P}^m(\mathbb{R})}$,
\begin{gather}\label{ineq:KimProjective}
    \bla(\mathbb{P}^m(\mathbb{R}), g) < 2^{\frac{2}{m}} \blo(\mathbb{P}^m(\mathbb{R})) = (2m+2) \omega_m^\frac{2}{m}.
\end{gather}
\end{cor}

On projective space, Kim aimed originally for a stronger inequality and so pursues a more explicit construction than in our paper. To complete her approach, the missing ingredient seems to be some new kind of ``folded conformal volume'', as is highlighted in Section 7 of her paper. As a first step in that direction, Kim also investigates limiting cases of the folded volume when either the fold map or the M\"{o}bius transformation degenerates. 
\subsubsection*{Complex and Quaternionic Projective Spaces}

The complex projective space of dimension $d$ is $\mathbb{P}^d(\mathbb{C}) \cong \S^{2d+1}/\S^1$. Its canonical metric is induced by this homogeneous representation. It is known as the Fubini-Study metric. Its volume is $\pi^d/d!$ and its first non-zero eigenvalue is $\lambda_1(\mathbb{P}^d(\mathbb{C})) = 4(d+1)$.

The quaternionic projective space of dimension $d$ is $\mathbb{P}^d(\mathbb{H}) \cong \S^{4d+3}/\S^3$. Its canonical metric has a volume of $\pi^{2d}/(2d+1)!$ and its first non-zero eigenvalue is $\lambda_1(\mathbb{P}^d(\mathbb{H})) = 8(d+1)$.

\subsubsection*{The Octonionic Projective Plane}

The octonionic projective plane of dimension 16 is $\mathbb{P}^2(\mathbb{O}) \cong \S^{23}/\S^7$. The volume of its canonical metric is $6 \pi^8/11!$, and its first non-zero eigenvalue is $\lambda_1(\mathbb{P}^2(\mathbb{O})) = 48$.

\subsubsection*{The 2-dimensional equilateral Torus}

The 2-dimensional equilateral torus \( \mathbb{T}^2_{\text{eq}} \) is a flat torus induced by the metric generated by the lattice \( \mathbb{Z}(1, 0) \oplus \mathbb{Z} \left(\frac{1}{2}, \frac{\sqrt{3}}{2}\right) \). Nadirashvili showed in \cite{Nadirashvili1996} that this is the unique metric that maximizes $\blo$, a result previously known by Berger \cite{Berger1973} within the class of flat tori. The volume of this torus is $\frac{\sqrt{3}}{2}$, and its first non-zero eigenvalue is $\lambda_1(\mathbb{T}^2_{eq}) = \frac{16 \pi^2}{3}$. Moreover, El Soufi and Ilias proved in \cite{ElSoufiIlias2000} that, together with the flat square torus $\mathbb{T}^2_{\text{square}} = \mathbb{R}^2 / \mathbb{Z}^2$, these are the only two metrics on a genus-one surface that admit a (minimal) isometric immersion via the first eigenfunctions into a sphere: $\S^3$
for the square torus and $\S^5$
for the equilateral torus.

\subsubsection*{The generalized Clifford Torus}

The generalized Clifford torus is
$\mathbb{S}^p\left(\sqrt{\frac{p}{m}}\right)\times \mathbb{S}^q\left(\sqrt{\frac{q}{m}}\right)$, where $p+q=m$, endowed with the canonical product metric. Its first eigenvalue is $\l_1\left(\mathbb{S}^p\left(\sqrt{\frac{p}{m}}\right)\times \mathbb{S}^q\left(\sqrt{\frac{q}{m}}\right) \right)=m $, and its volume is $(\frac{p^p q^q}{m^m})^{\frac{1}{2}}\omega_p\omega_q$.

\subsubsection*{The Klein Bottle}

Another illustrative example of a non-orientable manifold is the Klein bottle $\mathbb{K}$.
Let $g_0$ be the metric of revolution
$$g_0 = \frac{9 + (1+8 \cos^2 v)^2}{1+8 \cos^2 v} (du^2 + \frac{dv^2}{1+8 \cos^2 v}),$$
where $0 \leq u < \frac{\pi}{2}$ and $0 \leq v < \pi$.

It was proven by Jakobson, Nadirashvili, and Polterovich in \cite{jakobson2006extremal} that this metric admits a minimal isometric immersion into $\mathbb{S}^4$ via the first eigenfunctions, and hence $V_c(4,\mathbb{K}) = \text{Area}(\mathbb{K}, g_0) = 6\pi E(\frac{2\sqrt{2}}{3})$, where $E(\cdot)$ is a complete elliptic integral of the second kind.

They also showed that $(\mathbb{K}, g_0)$ is a bipolar surface of the Lawson torus $\tau_{3,1}$, and that $g_0$ is extremal for the functional defined on the space of Riemannian metrics of given area $g \to \lambda_1(\mathbb{K}, g)\text{Area}(\mathbb{K}, g)$. This provides the following upper bound:
$$\lambda_1(\mathbb{K}, g)\text{Area}(\mathbb{K}, g) \leq \lambda_1(\mathbb{K}, g_0)\text{Area}(\mathbb{K}, g_0) = 12 \pi E(\frac{2\sqrt{2}}{3}).$$
The question of the uniqueness of the maximal metric was later resolved by El Soufi, Giacomini, and Jazar \cite{ElSou-Giaco-Jazar}, who showed that $g_0$ is indeed the unique metric achieving equality.

\medskip

Each of these examples, endowed with their canonical metrics, achieves equality in Theorem \ref{thm:ElSoufiIlias}, providing an explicit formula for their conformal volume. Indeed, $\Sm$, $\mathbb{P}^m(\mathbb{R})$, $\mathbb{P}^2(\mathbb{O})$, $\mathbb{P}^d(\mathbb{C})$, and $\mathbb{P}^d(\mathbb{H})$ are all symmetric spaces of rank 1, and so they are also compact, irreducible, homogeneous spaces.

The generalized Clifford torus $\mathbb{S}^p\left(\sqrt{\frac{p}{m}}\right) \times \mathbb{S}^q\left(\sqrt{\frac{q}{m}}\right)$ admits a minimal immersion into $\S^{2m-1}$ by first eigenfunctions, see the paper by Park and Urakawa \cite{park1991classification}. 

The equilateral torus $\mathbb{T}^2_{\text{eq}}$ of dimension 2 is a result of Nadirashvili \cite{Nadirashvili1996}, and finally, the Klein bottle $\mathbb{K}$ is a result of Jakobson, Nadirashvili, and Polterovich \cite{jakobson2006extremal}, as previously mentioned.

\subsubsection*{Neumann problem}
The proof of Theorem~\ref{thm:thmprincipal} also applies to compact manifolds with boundary, leading in particular to upper bounds for the second non-zero Neumann eigenvalue $\mu_2$ of bounded Euclidean domains $\Omega\in\R^m$:
$$\mu_2(\Omega)\leq 2^{2/m}(m\omega_m)^{2/m}.$$
However this bound is not sharp, and it should be compared to that of Bucur and Henrot \cite{BucurHenrot}. This follows from the construction, which maps $\Omega$ to a sphere in an artificial way for this setting.

\subsection{Plan of the paper}
The proof of Theorem~\ref{thm:thmprincipal} is based on the construction of an admissible two-dimensional space of trial functions, which is obtained by combining the construction of the conformal volume $V_c(M,C)$ of Li and Yau~\cite{Li-Yau} with the folding mechanism that was introduced by the second author, Nadirashvili and Polterovich~\cite{GNP} in their study of the second eigenvalue $\lambda_2$ of spheres.

\section{\bf An upper bound on the second conformal eigenvalue}
Let $(f_j)_{j\geq 0}$ be an orthonormal basis of $L^2(M,g)$ corresponding to the eigenvalues $\lambda_j(M,g)$. We use one of the standard variational characterizations of $\lambda_k(M,g)$:
\begin{gather}\label{caravariation}
\lambda_k(M,g)=\min_{f \in A_k\setminus\{0\}}\frac{\int_M|\nabla f|^2\,dv_g}{\int_M f^2\,dv_g},
\end{gather}
where $A_k$ is the following subspace of the Sobolev space $H^1(M)$: 
$$A_k = \setc*{ f \in  H^1(M)} {\int_M ff_j dv_g=0 \text{ for }j=0,1,\cdots,k-1}.$$ 
Functions in $A_k$ are said to be \emph{admissible} and they are used as trial functions in~\eqref{caravariation} to provide upper bounds on $\lambda_k(M,g)$.

In order to prove Theorem~\ref{thm:thmprincipal}, trial functions for $\lambda_2$ will be constructed by combining the definition of the conformal volume of Li and Yau~\cite{Li-Yau} with the folding mechanism that was introduced in~\cite{GNP}. We start by recalling some well-known ideas from Hersch~\cite{Hersch}.

\subsection{\bf Center of mass arguments}
For each $s\in\R^{n+1}$ we define $X_s:\S^n\to\R$ by
$$X_s(p)=\langle s,p\rangle.$$
These are the coordinate functions on the sphere $\S^n$, and they form the eigenspace corresponding to $\lambda_1(\S^n,g_0)=n$.
Any conformal automorphism $\tau$ of $\S^n$ can be decomposed as
$$\tau=T\circ\phi_\xi,$$
where $T\in O(n+1)$ and for each $\xi\in\mathbb{B}^{n+1}\subset\R^{n+1}$, 
$\phi_\xi:\Sn\to\Sn$ is defined by
\begin{equation}\label{defiphi_xi}
    \phi_{\xi}(x)=\xi +\frac{1- |\xi|^2}{|x+\xi|^2}(x+\xi).
\end{equation}
These automorphisms $\phi_\xi$ are particularly nice and have been used extensively. Notice that $\phi_0=\text{id}_{\Sn}$ and for $\xi_0\in\S^n$
\begin{gather}\label{eq:limitxi}
    \lim_{\xi\to \xi_0}\phi_{\xi}(p)=\xi_0\qquad\forall p\neq -\xi_0.
\end{gather}
In other words, for $\|\xi\|$ near 1, the automorphism $\phi_\xi$ pushes most of the sphere to a small neighborhood of $\xi/\|\xi\|$.

Motivated by the work of Szeg\H{o}~\cite{Szego}, Hersch used this rich structure in~\cite{Hersch} to construct a trial function for $\l_1(\mathbb{S}^2,g)$.
One of the key ideas of his proof was to renormalize the center of mass of some measures on $\S^2$. We state a recent version by Laugesen~\cite[Corollary 5]{Laugesencenterofmass}. 
\begin{lemma}[Hersch's center of mass]\label{lemmeHersch}
Let $\mu$ be a Borel measure on the unit sphere $\Sn$, with $n \geq 1$, satisfying $0 < \mu(\Sn) < \infty.$ 
If for all $y \in \Sn,$
\begin{gather}\label{ineq:halfmass}
    \mu(\{y\}) <\frac{1}{2}\mu(\Sn)
\end{gather}
then a unique point $\xi=\xi(\mu) \in \mathbb{B}^{n+1}$ exists such that:
\begin{equation}\label{mesurenonatomique}
    \int_{\Sn}\phi_{\xi}\,d\mu=0.
\end{equation}
This point $\xi(\mu)$ depends continuously on the measure $\mu$. That is, if $\mu$ satisfies~\eqref{ineq:halfmass} and if $\mu_k \to \mu$ weakly, where each $\mu_k$ also satisfies~\eqref{ineq:halfmass}, then $\xi(\mu_k) \to \xi(\mu)$ as $k \to \infty$.
\end{lemma}
\begin{rem}
    In our applications, the measures $\mu$ will satisfy $\mu(\{y\})=0$ for each point $y$.
\end{rem}
\begin{rem}
    Let $\pi:\Sn\to\R^{n+1}$ be the canonical embedding. That is,
    $$\pi(x)=(X_{e_1}(x),\cdots,X_{e_{n+1}}(x)).$$
    The usual definition of the center of mass of a measure $\mu$ on $\Sn$ is
    $$\frac{1}{\mu(\Sn)}\int_{\Sn}\pi\,d\mu\in\mathbb{B}^{n+1}.$$
    Equation~\eqref{mesurenonatomique} can be written in terms of the push-forward measure $\nu={\phi_{\xi}}_\star\mu$:
    \begin{equation*}
     0=\int_{\Sn}\phi_{\xi}\,d\mu=\int_{\Sn}\pi\,d\nu.
    \end{equation*}
    In other words, the center of mass of the push-forward measure $\nu={\phi_{\xi}}_*\mu$ the origin.
    Some authors also call the renormalization parameter $\xi(\mu)$ the center of mass.
\end{rem}

\subsection{\bf Folding spherical caps}
The next important topological idea is called the folding method and consists of introducing spherical caps to construct a trial function satisfying two orthogonality conditions simultaneously, so that it will be admissible in the characterization~\eqref{caravariation} of $\lambda_2$. This folding method was first introduced by Nadirashvili in \cite{nadirashvili2002} to prove that $$\bla(\mathbb{S}^2,g)\leq 16 \pi.$$
The bound of $16\pi$ is achieved by a degenerating sequence of metrics into two disjoint spheres. In 2009, Girouard, Nadirashvili and Polterovich \cite{GNP} generalized this result to higher dimensional spheres with odd dimensions up to a asymtotically sharp constant. Later in 2014, some progress was made by Petrides \cite{Petrides2014} to deal with the case of even dimensions, until finally Kim \cite{Kim} managed to obtain the optimal result that we already stated in corollary \ref{Kimresultsphere}.

\subsubsection*{Construction of folding maps}
We start by presenting a natural parametrization of the space of all spherical caps $C\subset\Sn$.  Let $(p,t)\in \Sn \times (-1,1)$ and consider
the hemisphere 
$$C_{(p,0)}:=\{x \in \Sn\,:\, \langle x, p\rangle >0 \}$$ 
centered on $p\in\Sn$. The spherical cap $C_{(p,t)}$ centered at $p$ and with radius $t$ is defined as
$$C_{(p,t)} = \phi_{-tp}(C_{(p,0)}).$$
It follows from~\eqref{eq:limitxi} that
$$\lim_{t\to 1}C_{(p,t)}=\Sn\qquad\text{ and }\qquad\lim_{t\to -1}C_{(p,t)}=\{p\}.$$
Thus, the space of spherical caps is naturally identified with $(-1,1)\times\Sn$, and its natural compactification is $\overline{\mathbb{B}}^{n+1}$, where the origin $0$ corresponds to the limit as $t\to 1$.

Given $p\in\Sn$, the reflection $R_p:\Sn\to\Sn$ across the hyperplane perpendicular to $p$ and containing the origin is given by 
$$R_p(x)=x - 2\langle x,p \rangle p.$$
Conjugation with the automorphism defining a spherical cap $C$ allows the definition of reflections $\tau_C:\Sn\to\Sn$ across the boundary of any spherical cap $C=C_{(p,t)}$:
$$\tau_C:= \phi_{-tp} \circ R_p \circ \phi_{tp}.$$
For the spherical cap $C=C_{(p,t)}$ we  can finally define the folding map 
$F_C =F_{(p,t)}: \Sn \to C$ by 
$$
F_C(x)=
\begin{cases}
    x &\text{ if }x \in C, \\
    \tau_C(x) &\text{ otherwise }.
\end{cases}
$$

This construction was introduced in~\cite{GNP}. See also~\cite{Kim,BucurHenrot}

\subsection*{\bf General properties}
The following proposition summarizes some simple but useful properties that follow directly from the definitions provided above. These properties were previously discussed in~\cite[Lemma 5]{Kim} in the context of $\mathbb{S}^n$ with its volume measure $dv_g$, and also in~\cite[Lemma 5.2]{Freitas-Laugesentwoballs} for bounded domains of $\mathbb{R}^n$ endowed with a hyperbolic measure. For our purposes, we state these properties for an arbitrary Borel measure $\mu$ on $\mathbb{S}^n$ that satisfies the Hersch center of mass condition~\eqref{ineq:halfmass}. The proof can be derived in the same manner as in ~\cite[Lemma 5.2]{Kim}.

Let $\xi_C = \xi_{(p, t)}$ denote the center of mass of the push-forward measure ${F_C}_* \mu$, referred to as the renormalization point of the cap $C$.

\begin{prop}\label{propriétés}
     \begin{enumerate}[label=(\roman*)]
        \item The renormalization point $\xi_{(p,0)}$ of the hemisphere cap $C=C_{(p,0)}$ satisfies: 
        $$\xi_{(p,0)}=R_p(\xi_{(-p,0)}).$$
        \item  We also have $$F_{(p,0)}=R_p\circ F_{(-p,0)}.$$
        \item For any $\xi \in \mathbb{B}^{n+1}$,
        $$\phi_{\xi} \circ R_p = R_p \circ \phi_{R_p (\xi)}.$$
   \end{enumerate}
\end{prop}

For the sake of completeness, we also present a lemma from
Kim~\cite[Proposition $8$]{Kim}.
 \begin{lemma}\label{condition double}
Let $g$ be a metric on $\Sn\subset\R^{n+1}$ and consider $f_1\in \mathcal{C}^\infty (\Sn)$ an eigenfunction for $\l_1(\Sn,g)$. Then there exists $p\in\Sn$ and $t\in [0,1)$ for which the renormalization point $\xi_C$ of the cap $C=C_{(p,t)}$ satisfies the following:
\begin{enumerate}[label=(\roman*)]
     
 \item $$\int_{\Sn}\phi_{\xi_C} \circ F_C\,dv_g=0,$$
 
 \item $$\int_{\Sn} (\phi_{\xi_C} \circ F_C)  f_1\,dv_g=0.$$
  \end{enumerate}
\end{lemma}
\begin{rem}
    Observe that $t\geq 0$ in the previous lemma means that the cap $C=C_{(p,t)}$ contains the hemisphere $C_{(p,0)}$.
\end{rem}
Finally, we will need the following lemma using a particular symmetry to obtain nonzero degree of a map, its proof was initially given by Petrides \cite{Petrides2014} by local surgery, and later by Laugesen and Freitas \cite{Freitas-Laugesentwoballs} by global surgery.
\begin{lemma}\label{lemmePetrides}
    Let $n\geq 2$ and let $V : \Sn \to \Sn$ a continuous map such that:
    \begin{equation}
        V(-p)=R_p(V(p)),
    \end{equation}
    for all $p\in \Sn.$
    Then $$\text{deg}(V) \neq 0.$$
\end{lemma}

\subsection{\bf Proof of the main result}
Let us now proceed to the proof of the theorem \ref{thm:thmprincipal}. 
Consider a conformal immersion $\phi: M \to \Sn$. This induces the Radon measure $\mu:=\phi_*dv_g$ on $\Sn$, which (using Lemma~\ref{lemmeHersch}) we assume satisfies
\begin{gather}\label{supposeHersch}
    \int_{M}\phi\,dv_g=\int_{\S^n}\pi\,d\mu=0\in\R^{n+1}.
\end{gather}
We will construct trial functions for $\lambda_2(M,g)$ of the following form:
$$f_{C,i}:=X_{e_i}\circ \phi_{\xi_C} \circ F_C \circ \phi\,\qquad\text{ for }\qquad i=1,\cdots,n+1.$$ 
Set $f_C=(f_{C,1},\cdots,f_{C,n+1}):M\to\Sn\subset\mathbb{R}^{n+1}$ and
observe that
$$\int_{M}f_C\,dv_g=\int_{\S^n}\phi_{\xi_C}\,d({F_C}_*\mu).$$
Because the measure $\nu:={F_C}_*\mu$ satisfies $\nu(\{y\})=0$ for each $y\in\Sn$,
it follows from Lemma~\ref{lemmeHersch} that there exists
a unique renormalization vector $\xi_C\in\mathbb{B}^{n+1}$ such that 
$$\int_{M}f_{C}\,dv_g=0.$$
In other words, the trial functions $f_{C,i}$ satisfy the first admissibility condition $\int_{M}f_{C,i}\,dv_g=0$ in the variational characterization~\eqref{caravariation} of $\lambda_2$.
In the second step, we will see that there exists a spherical cap $C$ for which the trial functions $f_{C,i}$ also satisfy the second admissibility condition
$$\int_{M}f_{C,i}f_1\,dv_g=0.$$
We record all of this in the next Lemma.
\begin{lemma}\label{lemmefcttest}
Let $(M,g)$ be a Riemannian manifold of dimension $m$, and let $f_1\in \mathcal{C}^\infty (M)$ be an eigenfunction associated to the first eigenvalue $\l_1(M,g)$. For each spherical cap $C\subset\Sn$, there exists a unique $\xi_C\in\mathbb{B}^{n+1}$ such that:
$$\int_{M}\phi_{\xi_C}\circ F_C  \circ \phi\,dv_g=0.$$
Moreover, if $\int_{M}\phi f_1 dv_g\neq 0$ then there exists a cap
$C=C_{(p,t)}$, with $t\in [0,1)$, such that:
$$\int_{M}\phi_{\xi_C}\circ F_C \circ \phi.f_1\, dv_g=0.$$
\end{lemma}
\begin{proof}
The proof of the first part is in the paragraph before the statement.
To prove the existence of the cap $C$, we work by contradiction and suppose all spherical caps $C=C_{(p,t)}$ satisfy
$$\psi(p,t):=\int_{M}\phi_{\xi_C}\circ F_C \circ \phi.f_1 dv_g \neq 0.$$
Next we consider the function $H :\Sn\times [0,1)\to \Sn$ defined by :
$$H(p,t)=\frac{1}{\|\psi(p,t)\|}\psi(p,t).$$

Because the renormalization point $\xi(\mu)$ depends continuously on the measure $\mu$, the function $H$ is continuous.

Moreover, as $t \to 1$, we have $C_{(p,t)} \to \Sn$, thus $H$ extends to the boundary by
$$H(p,1) := \lim_{t\to 1}H(p,t)=\frac{1}{\|\int_{M} \phi.f_1 dv_g\|} \int_M \phi.f_1 dv_g.$$
The trick now consists in looking at what happens at $C_{(-p,0)}$.
By proposition (\ref{propriétés}), since 
$$\xi_{(p,0)}=R_p(\xi_{(-p,0)})\qquad\text{ and }\qquad F_{(p,0)}=R_p\circ F_{(-p,0)},$$ and that for any $\xi \in \B^{n+1}$, $ \phi_{\xi} \circ R_p = R_p \circ \phi_{R_p (\xi)}$, the function $H$ satisfies the following symmetry relation

\begin{equation}\label{symmetriePetrides}
       H(p,0)=R_p\circ \left(H(-p,0)\right).
\end{equation}
Indeed, by the linearity of $R_p$ and the properties that we have just mentioned,
\begin{align*}
\int_{M} \phi_{\xi_{(-p,0)}} \circ F_{(-p,0)} \circ \phi f_1 \, dv_g 
&= \int_{M}  \phi_{\xi_{(-p,0)}} \circ R_p \circ F_{(p,0)} \circ \phi f_1 \, dv_g \\
&= \int_{M} R_p \circ \phi_{R_p(\xi_{(-p,0)})} \circ F_{(p,0)} \circ \phi f_1 \, dv_g \\
&= R_p\left(\int_{M} \phi_{\xi_{(p,0)}} \circ F_{(p,0)} \circ \phi f_1 \, dv_g\right).
\end{align*}
Inequality \eqref{symmetriePetrides} follows by observing that $R_p$ is an isometry of $\R^{n+1}$, ie $\|R_p(x)\|=\|x \|$ for all $x\in \R^{n+1}$.

Once this symmetry identity is established, the reasoning proceeds by observing that $H$ represents a homotopy of spheres between $H(p, 1)$, which is independent of $p$, thereby having a homotopy degree of zero, and $H(p, 0)$, which is a function satisfying the given symmetry \eqref{symmetriePetrides}. The main contribution of Petrides’s paper \cite{Petrides2014} was to prove that this symmetry has a non-zero homotopy degree, as recalled in Lemma \ref{lemmePetrides}. This leads to a contradiction.
\end{proof}
Using Lemma \ref{lemmefcttest}, we can now complete the proof of Theorem \ref{thm:thmprincipal}.

 \begin{proof}[Proof of Theorem \ref{thm:thmprincipal}]
 Let $\phi : M \rightarrow \Sn$ be a conformal immersion satisfying~\eqref{supposeHersch}, and let  $(e_i)_i$ be an orthonormal basis of $\mathbb{R}^{n+1}$.

The proof is simplest when $\int_{M}\phi f_1 dv_g=0$, since then the functions $X_{e_i}\circ\phi$ with $i=1,2,\cdots,m+1$ can be used as trial functions for $\lambda_2$, and it follows from~\eqref{caravariation}  that
 $$\lambda_2(M,g)\int_{M}X_{e_i}^2\circ\phi\,dv_g\leq \int_M|\nabla (X_{e_i}\circ\phi)|^2\,dv_g.$$
 Summing over $i$ and using $\sum_{i=1}^{n+1} X_{e_i}^2\equiv 1$, it follows from the H\"older inequality that
 $$\lambda_2(M,g)\text{vol}(M,g)
 \leq \left( \int_M \left( \sum_{i=1}^{n+1} |\nabla X_{e_i} \circ \phi|^2 \right)^{{m/2}} \, dv_g \right)^{2/m} \text{vol}(M, g)^{1-2/m}.$$
 A standard computation shows that
$$\phi^*g_0=\frac{1}{m}\left(\sum_{i=1}^{n+1} |\nabla X_{e_i}\circ \phi |^2\right)\, g,$$
from which it follows that
$$dv_{\phi^*g_0}=\frac{1}{m^{m/2}}\left(\sum_{i=1}^{n+1} |\nabla X_{e_i}\circ \phi |^2\right)^{m/2}\,dv_g$$
and
so that
$$\int_M \left( \sum_{i=1}^{n+1} |\nabla X_{e_i} \circ \phi|^2 \right)^{{m/2}} \, dv_g=m^{m/2}\int_{M}dv_{\phi^*g_0}=m^{m/2}\text{vol}(\phi(M)).$$
It follows by substitution that
\begin{align*}
    \lambda_2(M,g)\text{vol}(M,g)^{2/m}
 &\leq m\text{vol}(\phi(M))^{2/m} \\
 &<2^{2/m}m\text{vol}(\phi(M))^{2/m}
 \leq 2^{2/m}m V_c(n,\phi)^{2/m}.
\end{align*}

In the situation when $\int_{M}\phi f_1 dv_g\neq 0$, the proof is more complicated. It requires the use of the folding mechanism.
 By the standard variational characterization (\ref{caravariation}) of $\lambda_2(M,g)$, lemma~\ref{lemmefcttest} shows the existence of a pair $(C,\xi_C)$ for which the folded function $f_{C}= \phi_{\xi_C}\circ F_c  \circ \phi$ is admissible for $\l_2(M,g)$.
 
 For all $i=1,\cdots, n+1$, we have
 $$\lambda_2(M,g) \int_{M}X^2_{e_i} \circ \phi_{\xi_C}\circ F_C  \circ \phi \,dv_g \leq\int_{M}|\nabla X_{e_i} \circ \phi_{\xi_C}\circ F_C\circ \phi |^2  \, dv_g .$$
 In the same fashion as in the previous case, by summing over the coordinate functions, we obtain 

  $$\lambda_2(M,g) \text{vol}(M,g) \leq\int_{M}\sum_{i=1}^{n+1}|\nabla X_{e_i} \circ \phi_{\xi_C}\circ F_C\circ \phi |^2 \,dv_g .$$
Next we apply  Hölder's inequality to get
\begin{align}\label{ineg après Holder}
    \lambda_2(M, g) \, \text{vol}(M, g) &\leq \left( \int_M \left( \sum_{i=1}^{n+1} |\nabla X_{e_i} \circ \phi_{\xi_C} \circ F_C \circ \phi|^2 \right)^{{m/2}} \, dv_g \right)^{2/m} \text{vol}(M, g)^{1-2/m}.
    \end{align}
By splitting the integral on the right-hand side over \( \phi^{-1}(C) \) and its complement in \( M \), and using the definition of the fold map \( F_C \), we find
\begin{align}\label{decom sur C}
   \int_{M}  \left(  \sum_{i=1}^{n+1} |\nabla X_{e_i} \circ \phi_{\xi_C} \circ F_C \circ \phi|^2 \right)^{{m/2}} \, dv_g &= 2 \int_{\phi^{-1}(C)} \left( \sum_{i=1}^{n+1} |\nabla X_{e_i} \circ \phi_{\xi_C} \circ \phi|^2 \right)^{{m/2}} \, dv_g \\\nonumber
   & < 2 \int_{M} \left( \sum_{i=1}^{n+1} |\nabla X_{e_i} \circ \phi_{\xi_C} \circ \phi|^2 \right)^{{m/2}} \, dv_g.
\end{align}
We then use the expression of the pull back metric $(\phi_{\xi_C} \circ  \phi)^*g_0$ that is given by 
  $$(\phi_{\xi_C} \circ  \phi)^*g_0= \frac{1}{m}\sum_{i=1}^{n+1} |\nabla X_{e_i} \circ \phi_{\xi_C} \circ  \phi |^2\, g$$ 
to deduce that its volume can be expressed as
\begin{equation}\label{express densité pull back}
\text{vol}((\phi_{\xi_C} \circ  \phi)^*g_0) =\frac{1}{m^{m/2}} \int_{M} \left( \sum_{i=1}^{n+1} |\nabla X_{e_i} \circ \phi_{\xi_C} \circ \phi|^2 \right)^{{m/2}} \, dv_g.
\end{equation}
Substituting equation \eqref{express densité pull back} and inequality \eqref{decom sur C} into inequality \eqref{ineg après Holder}, we obtain
\begin{align*}
   \bla(M,g)  &< 2^{2/m} m \, \text{vol}((\phi_{\xi_C} \circ \phi)^*g_0)^{{2/m}} \\ \notag
    &\leq 2^{{2/m}} m \, V_c(n, \phi)^{{2/m}} .
\end{align*} 
In all cases we have proven that
\begin{align}
   \bla(M,g)  <  2^{{2/m}} m \, V_c(n, \phi)^{{2/m}}.
\end{align} 
The proof of Theorem \ref{thm:thmprincipal} is completed by taking the infimum over the set of conformal immersions of $M$ into $\Sn$,
$$\bla(M,g)< 2^{2/m}mV_c(n,M,C)^{2/m}.$$

 \end{proof}

\bibliographystyle{plain}
\bibliography{biblio}
\bigskip
\end{document}